\documentclass[reqno]{amsart}

\usepackage{amssymb}
\usepackage{graphicx}
\usepackage{amscd}
\usepackage[hidelinks]{hyperref}
\usepackage{color}
\usepackage{float}
\usepackage{graphics,amsmath,amssymb}
\usepackage{amsthm}
\usepackage{amsfonts}
\usepackage{latexsym}
\usepackage{epsf}
\usepackage{xifthen}
\usepackage{mathrsfs}
\usepackage{dsfont}
\usepackage{makecell}
\usepackage[FIGTOPCAP]{subfigure}
\usepackage{amsmath}
\allowdisplaybreaks[4]
\usepackage{listings}
\usepackage{etoolbox}
\usepackage{fancyhdr}
\usepackage{pdflscape}
\usepackage[title,toc,titletoc]{appendix}
\usepackage{enumitem}
\usepackage[noadjust]{cite}
\usepackage{forest}

 \newtheoremstyle{mytheorem}
 {3pt}
 {3pt}
 {\slshape}
 {}
 {\bfseries}
 {.}
 { }
 {}

\numberwithin{equation}{section}

\theoremstyle{theorem}
\newtheorem{theorem}{Theorem}[section]
\newtheorem*{theorem*}{Theorem}
\newtheorem{corollary}[theorem]{Corollary}
\newtheorem{lemma}[theorem]{Lemma}

\newtheorem{innercustomgeneric}{\customgenericname}
\providecommand{\customgenericname}{}
\newcommand{\newcustomtheorem}[2]{%
	\newenvironment{#1}[1]
	{%
		\renewcommand\customgenericname{#2}%
		\renewcommand\theinnercustomgeneric{##1}%
		\innercustomgeneric
	}
	{\endinnercustomgeneric}
}
\newcustomtheorem{ctheorem}{Theorem}

\theoremstyle{definition}
\newtheorem{definition}{Definition}[section]

\newtheorem*{example*}{Example}

\theoremstyle{remark}

\newtheorem*{remark*}{Remark}
\newtheorem*{remarks*}{Remarks}

\newtheoremstyle{named}{}{}{\itshape}{}{\bfseries}{.}{.5em}{#1\thmnote{ #3}}
\theoremstyle{named}

\newcommand{\Keywords}[1]{\ifthenelse{\isempty{#1}}{}{\smallskip \smallskip \noindent \textbf{Keywords}. #1}}
\newcommand{\MSC}[2][2020]{\ifthenelse{\isempty{#2}}{}{\smallskip \smallskip \noindent \textbf{#1MSC}. #2}}
\newcommand{\abstractnote}[1]{\ifthenelse{\isempty{#1}}{}{\smallskip \smallskip \noindent \textsuperscript{\dag}#1}}

\makeatletter
\def\specialsection{\@startsection{section}{1}%
  \z@{\linespacing\@plus\linespacing}{.5\linespacing}%
  {\normalfont}}
\def\section{\@startsection{section}{1}%
  \z@{.7\linespacing\@plus\linespacing}{.5\linespacing}%
  {\normalfont\scshape}}
\patchcmd{\@settitle}{\uppercasenonmath\@title}{\Large\boldmath}{}{}
\patchcmd{\@settitle}{\begin{center}}{\begin{flushleft}}{}{}
\patchcmd{\@settitle}{\end{center}}{\end{flushleft}}{}{}
\patchcmd{\@setauthors}{\MakeUppercase}{\normalsize}{}{}
\patchcmd{\@setauthors}{\centering}{\raggedright}{}{}
\patchcmd{\section}{\scshape}{\large\bfseries\boldmath}{}{}
\patchcmd{\subsection}{\bfseries}{\bfseries\boldmath}{}{}
\renewcommand{\@secnumfont}{\bfseries}
\patchcmd{\@startsection}{\@afterindenttrue}{\@afterindentfalse}{}{}
\patchcmd{\abstract}{\leftmargin3pc}{\leftmargin1pc}{}{}

\def\maketitle{\par
  \@topnum\z@ 
  \@setcopyright
  \thispagestyle{empty}
  \ifx\@empty\shortauthors \let\shortauthors\shorttitle
  \else \andify\shortauthors
  \fi
  \@maketitle@hook
  \begingroup
  \@maketitle
  \toks@\@xp{\shortauthors}\@temptokena\@xp{\shorttitle}%
  \toks4{\def\\{ \ignorespaces}}
  \edef\@tempa{%
    \@nx\markboth{\the\toks4
      \@nx\MakeUppercase{\the\toks@}}{\the\@temptokena}}%
  \@tempa
  \endgroup
  \c@footnote\z@
  \@cleartopmattertags
}
\makeatother

\newcommand{\tF}{\tilde{F}}
\newcommand{\tG}{\tilde{G}}
\newcommand{\tPhi}{\tilde{\Phi}}
\newcommand{\tphi}{\tilde{\phi}}
\newcommand{\tPsi}{\tilde{\Psi}}
\newcommand{\tpsi}{\tilde{\psi}}

\newcommand{\fC}{\mathfrak{C}}
\newcommand{\fS}{\mathfrak{S}}

\newcommand{\drop}{\operatorname{drop}}
\newcommand{\DROP}{\operatorname{DROP}}
\newcommand{\ecyc}{\sim_{\scriptscriptstyle\operatorname{cyc}}}
\newcommand{\card}{\operatorname{card}}

\newcommand{\ddd}{\operatorname{d}}

\title{Parity considerations for drops in cycles on $\{1,2,\ldots,n\}$}

\author[S. Chern]{Shane Chern}
\address{Department of Mathematics and Statistics, Dalhousie University, Halifax, NS, B3H 4R2, Canada}
\email{chenxiaohang92@gmail.com}

\date{}

\begin{document}

\maketitle

\begin{abstract}

In 2019, A. Lazar and M. L. Wachs conjectured that the number of cycles on $[2n]$ with only even-odd drops equals the $n$-th Genocchi number. In this paper, we restrict our attention to a subset of cycles on $[n]$ that in all drops in the cycle, the latter entry is odd. We deduce two bivariate generating functions for such a subset of cycles with an extra variable introduced to count the number of odd-odd and even-odd drops, respectively. One of the generating function identities confirms Lazar and Wachs' conjecture, while the other identity implies that the number of cycles on $[2n-1]$ with only odd-odd drops equals the $(n-2)$-th Genocchi median.


\MSC{05A05, 05A15, 35C10.}
\end{abstract}

\section{Introduction}

Let $\fS_n$ denote the set of permutations on $\{1,2,\ldots,n\}=:[n]$. There is a natural equivalence relation $\ecyc$ on $\fS_n$ defined as follows. For $\pi=\pi_1\pi_2\cdots \pi_n$ and $\tilde{\pi}=\tilde{\pi}_1\tilde{\pi}_2\cdots \tilde{\pi}_n$, two permutations in $\fS_n$, we say $\pi\ecyc \tilde{\pi}$ if there exists a nonnegative integer $m$ such that $\pi_i=\tilde{\pi}_{m+i}$ for all $1\le i\le n$; here we assume that $\tilde{\pi}_K=\tilde{\pi}_k$ where $1\le k\le n$ is the unique index such that $K\equiv k \pmod{n}$. In other words, we treat the entries in a permutation in a cyclic way. From this perspective, we may define cycles on $[n]$ as equivalence classes in the quotient set $\fS_n/\!\ecyc$.

\begin{definition}
	Let $\ecyc$ be the equivalence relation defined as above.
	\begin{enumerate}[label=(\roman*)., widest=ii., itemindent=*, leftmargin=*]
		\item We say $[\pi]$ is a \textit{cycle} on $\{1,2,\ldots,n\}$ if $[\pi]$ is a member of the quotient set $\fS_n/\!\ecyc$. We also denote the quotient set $\fS_n/\!\ecyc$ by $\fC_n$, the set of cycles on $[n]$.
		
		\item For each $[\pi]\in\fC_n$, we assume that the representative $\pi=\pi_1\pi_2\cdots \pi_n$ is a permutation in $\fS_n$ that starts with $\pi_1=1$. Also, for any positive index $K$, we put $\pi_K=\pi_k$ where $1\le k\le n$ is the unique index such that $K\equiv k \pmod{n}$.
	\end{enumerate}
\end{definition}

Our starting point is a conjecture of A. Lazar and M. L. Wachs \cite[Conjecture 6.4]{LW2019} on cycles with only even-odd drops; this conjecture was recently proved by Z. Lin and S. H. F. Yan \cite{LY2022} through a bijective approach, and independently by Q. Pan and J. Zeng \cite{PZ2021} using continued fractions.

Here we say a consecutive pair $(\pi_i,\pi_{i+1})$ (with $1\le i\le n$) is a \textit{drop} in a cycle $[\pi]\in\fC_n$ if $\pi_i>\pi_{i+1}$. Notice that drops in cycles will not be affected by the choice of the representative of each equivalence class in $\fC_n$. We also say a drop $(\pi_i,\pi_{i+1})$ is \textit{odd-odd} (resp.~\textit{even-odd}) if $\pi_i$ is odd (resp.~$\pi_i$ is even) and $\pi_{i+1}$ is odd.

Let
$$\DROP([\pi]):=\{(\pi_i,\pi_{i+1}):\,\text{$\pi_i>\pi_{i+1}$ with $1\le i\le n$}\},$$
the set of drops in $[\pi]$. Further, for the unique cycle $[(1)]$ in $\fC_1$, we assume that it has a unique drop:
$$\DROP([(1)])=\{(\star,1)\},$$
where we tactically assume that the parity of ``$\star$'' is neither even nor odd. Now, for each cycle $[\pi]\in\fC_n$, we define
\begin{align*}
\drop_{oo}([\pi])&:=\card\{(\pi_i,\pi_{i+1})\in\DROP([\pi]):\,\text{$(\pi_i,\pi_{i+1})$ is odd-odd}\},\\
\drop_{eo}([\pi])&:=\card\{(\pi_i,\pi_{i+1})\in\DROP([\pi]):\,\text{$(\pi_i,\pi_{i+1})$ is even-odd}\}.
\end{align*}

Since each $[\pi]\in\fC_n$ must have a drop of the form $(a,1)$ for some $a$, our attention is then restricted to a subset of $\fC_n$:
$$\fC_n^o:=\{\pi\in\fC_n:\,\text{$\pi_{i+1}$ is odd for all drops $(\pi_i,\pi_{i+1})$ in $[\pi]$}\}.$$
We remark that the set $\fC_n^o$ includes cycles with only even-odd drops in the conjecture of Lazar and Wachs as a subset.

The first object of this paper is the following bivariate generating function identity.

\begin{theorem}\label{th:oo}
	\begin{align}
	&\sum_{n\ge 1}\sum_{[\pi]\in\fC_n^o}x^{\drop_{oo}([\pi])}t^n\notag\\
	& = \sum_{m\ge 1}\frac{m!(m-1)!t^{2m}}{\displaystyle \prod_{k=1}^m \big(1+k^2(1-x)t^2\big)}+\sum_{m\ge 1}\frac{\big((m-1)!\big)^2t^{2m-1}}{\displaystyle \prod_{k=1}^m \big(1+k^2(1-x)t^2\big)}.
	\end{align}
\end{theorem}

Recall from \cite{BD1981} that the (unsigned) \textit{Genocchi numbers}
$$\{g_n\}_{n\ge 1}=\{1, 1, 3, 17, 155, 2073, 38227, 929569,\ldots\},$$
which enumerate the number of Dumont permutations on $[2n-2]$, are given by the generating function
$$\sum_{n\ge 1}g_n t^n = \sum_{m\ge 1}\frac{m!(m-1)!t^{m}}{\prod_{k=1}^m \big(1+k^2 t\big)}.$$
Letting $x=0$ in Theorem \ref{th:oo} immediately yields an alternative proof of Lazar and Wachs' conjecture \cite[Conjecture 6.4]{LW2019}.

\begin{corollary}
	For all $n \ge 1$, $g_n$ is equal to the number of cycles on $[2n]$ with only even-odd drops.
\end{corollary}

Analogously, we may introduce another variable $y$ to count the number of even-odd drops.

\begin{theorem}\label{th:eo}
	\begin{align}
	&\sum_{n\ge 1}\sum_{[\pi]\in\fC_n^o}y^{\drop_{eo}([\pi])}t^n\notag\\
	&= (y-1)t^2+ \sum_{m\ge 1}\frac{m!(m-1)!t^{2m}}{\displaystyle \prod_{k=1}^m \big(1+k(k+1)(1-y)t^2\big)}+\sum_{m\ge 1}\frac{\big((m-1)!\big)^2t^{2m-1}}{\displaystyle \prod_{k=1}^m \big(1+k(k-1)(1-y)t^2\big)}.
	\end{align}
\end{theorem}

Recall also from \cite{BD1981} that the (unsigned) \textit{Genocchi medians}
$$\{h_n\}_{n\ge 0}=\{1, 2, 8, 56, 608, 9440, 198272, 5410688,\ldots\},$$
which enumerate the number of Dumont derangements on $[2n+2]$, are given by the generating function
$$1+\sum_{n\ge 1}h_{n-1} t^n = t^{-1}\sum_{m\ge 1}\frac{\big((m-1)!\big)^2t^{m}}{\prod_{k=1}^m \big(1+k(k-1)t\big)}.$$
Therefore, letting $y=0$ in Theorem \ref{th:eo}, we have the following result.

\begin{corollary}
	For all $n \ge 2$, $h_{n-2}$ is equal to the number of cycles on $[2n-1]$ with only odd-odd drops.
\end{corollary}

\section{A generating tree}

For $n\ge 2$, we consider an arbitrary cycle $[\pi]$ in $\fC_n^o$. First, in $\pi$, $n$ must be placed right before an odd entry. Removing $n$ from $\pi$, we are led to a cycle on $[n-1]$. It is also obvious that this cycle is in $\fC_{n-1}^o$. This means that all cycles in $\fC_n^o$ can be generated by cycles in $\fC_{n-1}^o$.

Now, given any cycle in $\fC_{n-1}^o$, it generates $\lfloor\frac{n+1}{2}\rfloor$ distinct cycles in $\fC_n^o$ by inserting $n$ right before an odd entry.

Assume that $n$ is even. If $n$ is inserted to the middle of an even-odd drop, then the numbers of even-odd and odd-odd drops remain the same; if $n$ is inserted to the middle of an odd-odd drop, then the number of even-odd drops increases by $1$ and the number of odd-odd drops decreases by $1$; otherwise, the number of even-odd drops increases by $1$ and the number of odd-odd drops remains the same.

Next, assume that $n$ is odd. If $n$ is inserted to the middle of an odd-odd drop, then the numbers of odd-odd and even-odd drops remain the same; if $n$ is inserted to the middle of an even-odd drop, then the number of odd-odd drops increases by $1$ and the number of even-odd drops decreases by $1$; otherwise, the number of odd-odd drops increases by $1$ and the number of even-odd drops remains the same.

The above arguments lead to a generating tree:

\begin{enumerate}[label=(\roman*)., widest=ii., itemindent=*, leftmargin=*]
	\item $\fC_{2n-1}^o\to \fC_{2n}^o$. Let $[\pi]$ be in $\fC_{2n-1}^o$ with $\drop_{oo}([\pi])=i$ and $\drop_{eo}([\pi])=j$.
	\begin{align*}
	[\pi]\mapsto\begin{cases}
	\text{$i$ cycles with $(i-1)$ odd-odd drops and $(j+1)$ even-odd drops},\\[8pt]
	\text{$j$ cycles with $i$ odd-odd drops and $j$ even-odd drops},\\[8pt]
	\text{$(n-i-j)$ cycles with $i$ odd-odd drops and $(j+1)$ even-odd drops}.
	\end{cases}
	\end{align*}
	
	\item $\fC_{2n}^o\to \fC_{2n+1}^o$. Let $[\pi]$ be in $\fC_{2n}^o$ with $\drop_{oo}([\pi])=i$ and $\drop_{eo}([\pi])=j$.
	\begin{align*}
	[\pi]\mapsto\begin{cases}
	\text{$i$ cycles with $i$ odd-odd drops and $j$ even-odd drops},\\[8pt]
	\text{$j$ cycles with $(i+1)$ odd-odd drops and $(j-1)$ even-odd drops},\\[8pt]
	\text{$(n-i-j)$ cycles with $(i+1)$ odd-odd drops and $j$ even-odd drops}.
	\end{cases}
	\end{align*}
\end{enumerate}

In other words, the following result holds true.

\begin{lemma}\label{le:g-tree}
	For $n\ge 1$,
	\begin{enumerate}[label=(\roman*)., widest=ii., itemindent=*, leftmargin=*]
		\item
		\begin{align*}
		&\sum_{[\pi]\in\fC_{2n}^o}x^{\drop_{oo}([\pi])}y^{\drop_{eo}([\pi])}\\
		&=\sum_{[\pi']\in\fC_{2n-1}^o}x^{\drop_{oo}([\pi'])}y^{\drop_{eo}([\pi'])}\\
		&\quad\times\big[\drop_{oo}([\pi'])x^{-1}y+\drop_{eo}([\pi'])+\big(n-\drop_{oo}([\pi'])-\drop_{eo}([\pi'])\big)y\big].
		\end{align*}
		
		\item
		\begin{align*}
		&\sum_{[\pi]\in\fC_{2n+1}^o}x^{\drop_{oo}([\pi])}y^{\drop_{eo}([\pi])}\\
		&=\sum_{[\pi']\in\fC_{2n}^o}x^{\drop_{oo}([\pi'])}y^{\drop_{eo}([\pi'])}\\
		&\quad\times\big[\drop_{oo}([\pi'])+\drop_{eo}([\pi'])xy^{-1}+\big(n-\drop_{oo}([\pi'])-\drop_{eo}([\pi'])\big)x\big].
		\end{align*}
	\end{enumerate}
\end{lemma}

\section{Odd-odd drops}

Letting $y=1$ in Lemma \ref{le:g-tree} yields the following relations.

\begin{lemma}\label{le:oo}
	Let
	\begin{align*}
	f_n = f_n(x):= \sum_{[\pi]\in\fC_{n}^o}x^{\drop_{oo}([\pi])}.
	\end{align*}
	Then $f_1=1$ and for $n\ge 1$,
	\begin{enumerate}[label=(\roman*)., widest=ii., itemindent=*, leftmargin=*]
		\item
		\begin{align*}
		f_{2n}=nf_{2n-1}-x\frac{\ddd}{\ddd x}f_{2n-1}+\frac{\ddd}{\ddd x}f_{2n-1}.
		\end{align*}
		
		\item
		\begin{align*}
		f_{2n+1}=nxf_{2n}-x^2\frac{\ddd}{\ddd x}f_{2n}+x\frac{\ddd}{\ddd x}f_{2n}.
		\end{align*}
	\end{enumerate}
\end{lemma}

\subsection{Cycles in $\fC_{2n}^o$}

By Lemma \ref{le:oo}, we have $f_2=1$ and for $n\ge 1$,
\begin{align*}
f_{2n+2}&=n^2 \cdot xf_{2n}\\
&\quad+n\Big(f_{2n}+2x\frac{\ddd}{\ddd x}f_{2n}-2x^2 \frac{\ddd}{\ddd x}f_{2n}\Big)\\
&\quad+\Big(\frac{\ddd}{\ddd x}f_{2n}-2x\frac{\ddd}{\ddd x}f_{2n}+x^2\frac{\ddd}{\ddd x}f_{2n}\\
&\quad\qquad+x\frac{\ddd^2}{\ddd x^2}f_{2n}-2x^2\frac{\ddd^2}{\ddd x^2}f_{2n}+x^3\frac{\ddd^2}{\ddd x^2}f_{2n}\Big).
\end{align*}
Defining
\begin{align*}
F=F(x,t):=\sum_{n\ge 1} f_{2n}t^n,
\end{align*}
the above recurrence relation then yields a PDE:
\begin{align}\label{eq:PDE-F}
t^{-1}(F-t)&=x(1-x)^2 \frac{\partial^2}{\partial x^2}F +2x(1-x)t \frac{\partial^2}{\partial x \partial t}F+xt^2 \frac{\partial^2}{\partial t^2}F\notag\\
&\quad +(1-x)^2 \frac{\partial}{\partial x}F + (1+x)t \frac{\partial}{\partial t}F.
\end{align}
Notice that $F(x,t)$ is also the unique power series solution of the above PDE at $t=0$.

Making the following change of variables
\begin{align*}
\begin{cases}
\xi=\dfrac{1}{1-x}\\[10pt]
\eta = \sqrt{(1-x)t}
\end{cases} \qquad\Longleftrightarrow\qquad
\begin{cases}
x=1-\dfrac{1}{\xi}\\[10pt]
t = \xi \eta^2
\end{cases}
\end{align*}
and using the chain rule (see, for instance, \cite[p.~65]{PR2005}), we are led to
\begin{align}\label{eq:PDE-Phi}
\eta^{-2}\Phi-\xi = \xi^2(\xi-1)\frac{\partial^2}{\partial \xi^2}\Phi + \xi(2\xi-1)\frac{\partial}{\partial \xi}\Phi,
\end{align}
where
$$\Phi=\Phi(\xi,\eta):=F\big(1-\tfrac{1}{\xi},\xi \eta^2\big)=F(x,t).$$

Recall that $F(x,t)$ is the unique power series solution of \eqref{eq:PDE-F} at $t=0$, then $\Phi(\xi,\eta)$ is also the unique power series solution of \eqref{eq:PDE-Phi} at $\eta=0$. Writing
$$\Phi(\xi,\eta):=\sum_{n\ge 0}u_n \eta^n$$
where $u_n:=u_n(\xi)$, then by \eqref{eq:PDE-Phi}, $u_0=0$, $u_1=0$, $u_2=\xi$, and for $n\ge 1$,
$$u_{n+2}=\xi^2(\xi-1)\frac{\ddd^2}{\ddd \xi^2}u_n + \xi(2\xi-1) \frac{\ddd}{\ddd \xi}u_n.$$
Thus, $u_n\in\mathbb{Z}[\xi]$ for all $n\ge 0$. This implies that $\Phi(\xi,\eta)$ is in $\mathbb{Z}[\xi][[\eta]]$, and therefore in $\mathbb{Z}[[\eta]][[\xi]]$.

Now, we may write
$$\Phi(\xi,\eta):=\sum_{m\ge 0}\phi_m \xi^m,$$
where $\phi_m:=\phi_m(\eta)$. From \eqref{eq:PDE-Phi}, we find that
\begin{gather*}
\phi_0=0,\\
\phi_1=\frac{\eta^2}{1+\eta^2},
\end{gather*}
and for $m\ge 2$,
\begin{align*}
\phi_m=\frac{m(m-1)\eta^2}{1+m^2\eta^2}\phi_{m-1}.
\end{align*}
It follows that, for $m\ge 1$,
\begin{align*}
\phi_m=\frac{m!(m-1)!\eta^{2m}}{\prod_{k=1}^m (1+k^2\eta^2)}.
\end{align*}

We conclude that
\begin{align*}
\Phi(\xi,\eta)=\sum_{m\ge 1}\frac{m!(m-1)!\eta^{2m}}{\prod_{k=1}^m (1+k^2\eta^2)}\xi^m,
\end{align*}
and therefore,
\begin{align}\label{eq:F(x,t)}
F(x,t)=\sum_{m\ge 1}\frac{m!(m-1)!t^{m}}{\prod_{k=1}^m \big(1+k^2(1-x)t\big)}.
\end{align}

\subsection{Cycles in $\fC_{2n-1}^o$}

We deduce from Lemma \ref{le:oo} that $f_1=1$ and for $n\ge 1$,
\begin{align*}
f_{2n+1}&=n^2 \cdot xf_{2n-1}\\
&\quad+n\Big(2x\frac{\ddd}{\ddd x}f_{2n-1}-2x^2 \frac{\ddd}{\ddd x}f_{2n-1}\Big)\\
&\quad+\Big({-x}\frac{\ddd}{\ddd x}f_{2n-1}+x^2\frac{\ddd}{\ddd x}f_{2n-1}\\
&\quad\qquad+x\frac{\ddd^2}{\ddd x^2}f_{2n-1}-2x^2\frac{\ddd^2}{\ddd x^2}f_{2n-1}+x^3\frac{\ddd^2}{\ddd x^2}f_{2n-1}\Big).
\end{align*}
Defining
\begin{align*}
\tF=\tF(x,t):=\sum_{n\ge 1} f_{2n-1}t^n,
\end{align*}
the above recurrence relation then yields a PDE:
\begin{align}\label{eq:PDE-tF}
t^{-1}(\tF-t)&=x(1-x)^2 \frac{\partial^2}{\partial x^2}\tF +2x(1-x)t \frac{\partial^2}{\partial x \partial t}\tF+xt^2 \frac{\partial^2}{\partial t^2}\tF\notag\\
&\quad -x(1-x) \frac{\partial}{\partial x}\tF + xt \frac{\partial}{\partial t}\tF.
\end{align}
Notice that $\tF(x,t)$ is also the unique power series solution of the above PDE at $t=0$.

Making the following change of variables
\begin{align*}
\begin{cases}
\xi=\dfrac{1}{1-x}\\[10pt]
\eta = \sqrt{(1-x)t}
\end{cases} \qquad\Longleftrightarrow\qquad
\begin{cases}
x=1-\dfrac{1}{\xi}\\[10pt]
t = \xi \eta^2
\end{cases}
\end{align*}
and using the chain rule, we have
\begin{align}\label{eq:PDE-tPhi}
\eta^{-2}\tPhi-\xi = \xi^2(\xi-1)\frac{\partial^2}{\partial \xi^2}\tPhi + \xi(\xi-1)\frac{\partial}{\partial \xi}\tPhi,
\end{align}
where
$$\tPhi=\tPhi(\xi,\eta):=\tF\big(1-\tfrac{1}{\xi},\xi \eta^2\big)=\tF(x,t).$$
Further, we observe that $\tPhi(\xi,\eta)$ is in $\mathbb{Z}[\xi][[\eta]]$, and therefore in $\mathbb{Z}[[\eta]][[\xi]]$.

We then write
$$\tPhi(\xi,\eta):=\sum_{m\ge 0}\tphi_m \xi^m,$$
where $\tphi_m:=\tphi_m(\eta)$. From \eqref{eq:PDE-tPhi}, we find that
\begin{gather*}
\tphi_0=0,\\
\tphi_1=\frac{\eta^2}{1+\eta^2},
\end{gather*}
and for $m\ge 2$,
\begin{align*}
\tphi_m=\frac{(m-1)^2\eta^2}{1+m^2\eta^2}\tphi_{m-1}.
\end{align*}
It follows that, for $m\ge 1$,
\begin{align*}
\tphi_m=\frac{\big((m-1)!\big)^2\eta^{2m}}{\prod_{k=1}^m (1+k^2\eta^2)}.
\end{align*}

We conclude that
\begin{align*}
\tPhi(\xi,\eta)=\sum_{m\ge 1}\frac{\big((m-1)!\big)^2\eta^{2m}}{\prod_{k=1}^m (1+k^2\eta^2)}\xi^m,
\end{align*}
and therefore,
\begin{align}\label{eq:tF(x,t)}
\tF(x,t)=\sum_{m\ge 1}\frac{\big((m-1)!\big)^2t^{m}}{\prod_{k=1}^m \big(1+k^2(1-x)t\big)}.
\end{align}

\subsection{Proof of Theorem \ref{th:oo}}

We observe that
\begin{align*}
\sum_{n\ge 1}\sum_{[\pi]\in\fC_n^o}x^{\drop_{oo}([\pi])}t^n&=\sum_{n\ge 1}f_{2n}t^{2n}+\sum_{n\ge 1}f_{2n-1}t^{2n-1}\\
&=F(x,t^2)+t^{-1}\tF(x,t^2).
\end{align*}
Substituting \eqref{eq:F(x,t)} and \eqref{eq:tF(x,t)} into the above gives Theorem \ref{th:oo}.

\section{Even-odd drops}

Letting $x=1$ in Lemma \ref{le:g-tree}, we arrive at relations as follows.

\begin{lemma}\label{le:eo}
	Let
	\begin{align*}
	g_n = g_n(y):= \sum_{[\pi]\in\fC_{n}^o}y^{\drop_{eo}([\pi])}.
	\end{align*}
	Then $g_1=1$ and for $n\ge 1$,
	\begin{enumerate}[label=(\roman*)., widest=ii., itemindent=*, leftmargin=*]
		\item
		\begin{align*}
		g_{2n}=nyg_{2n-1}-y^2\frac{\ddd}{\ddd y}g_{2n-1}+y\frac{\ddd}{\ddd y}g_{2n-1}.
		\end{align*}
		
		\item
		\begin{align*}
		g_{2n+1}=ng_{2n}-y\frac{\ddd}{\ddd y}g_{2n}+\frac{\ddd}{\ddd y}g_{2n}.
		\end{align*}
	\end{enumerate}
\end{lemma}

\subsection{Cycles in $\fC_{2n}^o$}

By Lemma \ref{le:eo}, we have $g_2=y$ and for $n\ge 1$,
\begin{align*}
g_{2n+2}&=n^2 \cdot yg_{2n}\\
&\quad+n\Big(yg_{2n}+2y\frac{\ddd}{\ddd y}g_{2n}-2y^2 \frac{\ddd}{\ddd y}g_{2n}\Big)\\
&\quad+\Big(y\frac{\ddd^2}{\ddd y^2}g_{2n}-2y^2\frac{\ddd^2}{\ddd y^2}g_{2n}+y^3\frac{\ddd^2}{\ddd y^2}g_{2n}\Big).
\end{align*}
Defining
\begin{align*}
G=G(y,t):=\sum_{n\ge 1} g_{2n}t^n,
\end{align*}
the above recurrence relation then yields a PDE:
\begin{align}\label{eq:PDE-G}
t^{-1}(G-yt)&=y(1-y)^2 \frac{\partial^2}{\partial y^2}G +2y(1-y)t \frac{\partial^2}{\partial y \partial t}G+yt^2 \frac{\partial^2}{\partial t^2}G+2yt \frac{\partial}{\partial t}G.
\end{align}
Notice that $G(y,t)$ is also the unique power series solution of the above PDE at $t=0$.

Making the following change of variables
\begin{align*}
\begin{cases}
\xi=\dfrac{1}{1-y}\\[10pt]
\eta = \sqrt{(1-y)t}
\end{cases} \qquad\Longleftrightarrow\qquad
\begin{cases}
y=1-\dfrac{1}{\xi}\\[10pt]
t = \xi \eta^2
\end{cases}
\end{align*}
and using the chain rule, we have
\begin{align}\label{eq:PDE-Psi}
\eta^{-2}\Psi-\xi +1 = \xi^2(\xi-1)\frac{\partial^2}{\partial \xi^2}\Psi + 2\xi(\xi-1)\frac{\partial}{\partial \xi}\Psi,
\end{align}
where
$$\Psi=\Psi(\xi,\eta):=G\big(1-\tfrac{1}{\xi},\xi \eta^2\big)=G(y,t).$$
Further, we observe that $\Psi(\xi,\eta)$ is in $\mathbb{Z}[\xi][[\eta]]$, and therefore in $\mathbb{Z}[[\eta]][[\xi]]$.

Now, we write
$$\Psi(\xi,\eta):=\sum_{m\ge 0}\psi_m \xi^m,$$
where $\psi_m:=\psi_m(\eta)$. From \eqref{eq:PDE-Psi}, we find that
\begin{gather*}
\psi_0=-\eta^2,\\
\psi_1=\frac{\eta^2}{1+2\eta^2},
\end{gather*}
and for $m\ge 2$,
\begin{align*}
\psi_m=\frac{m(m-1)\eta^2}{1+m(m+1)\eta^2}\psi_{m-1}.
\end{align*}
It follows that, for $m\ge 1$,
\begin{align*}
\psi_m=\frac{m!(m-1)!\eta^{2m}}{\prod_{k=1}^m \big(1+k(k+1)\eta^2\big)}.
\end{align*}

We conclude that
\begin{align*}
\Psi(\xi,\eta)=-\eta^2+\sum_{m\ge 1}\frac{m!(m-1)!\eta^{2m}}{\prod_{k=1}^m \big(1+k(k+1)\eta^2\big)}\xi^m,
\end{align*}
and therefore,
\begin{align}\label{eq:G(y,t)}
G(y,t)=(y-1)t+\sum_{m\ge 1}\frac{m!(m-1)!t^{m}}{\prod_{k=1}^m \big(1+k(k+1)(1-y)t\big)}.
\end{align}

\subsection{Cycles in $\fC_{2n-1}^o$}

It follows from Lemma \ref{le:eo} that $g_1=1$ and for $n\ge 1$,
\begin{align*}
g_{2n+1}&=n^2 \cdot yg_{2n-1}\\
&\quad+n\Big(g_{2n-1}-yg_{2n-1}+2y\frac{\ddd}{\ddd y}g_{2n-1}-2y^2 \frac{\ddd}{\ddd y}g_{2n-1}\Big)\\
&\quad+\Big(\frac{\ddd}{\ddd y}g_{2n-1}-3y\frac{\ddd}{\ddd y}g_{2n-1}+2y^2\frac{\ddd}{\ddd y}g_{2n-1}\\
&\quad\qquad+y\frac{\ddd^2}{\ddd y^2}g_{2n-1}-2y^2\frac{\ddd^2}{\ddd y^2}g_{2n-1}+y^3\frac{\ddd^2}{\ddd y^2}g_{2n-1}\Big).
\end{align*}
Defining
\begin{align*}
\tG=\tG(y,t):=\sum_{n\ge 1} g_{2n-1}t^n,
\end{align*}
the above recurrence relation then yields a PDE:
\begin{align}\label{eq:PDE-tG}
t^{-1}(\tG-t)&=y(1-y)^2 \frac{\partial^2}{\partial y^2}\tG +2y(1-y)t \frac{\partial^2}{\partial y \partial t}\tG+yt^2 \frac{\partial^2}{\partial t^2}\tG\notag\\
&\quad +(1-y)(1-2y) \frac{\partial}{\partial y}\tG+t \frac{\partial}{\partial t}\tG.
\end{align}
Notice that $\tG(y,t)$ is also the unique power series solution of the above PDE at $t=0$.

Making the following change of variables
\begin{align*}
\begin{cases}
\xi=\dfrac{1}{1-y}\\[10pt]
\eta = \sqrt{(1-y)t}
\end{cases} \qquad\Longleftrightarrow\qquad
\begin{cases}
y=1-\dfrac{1}{\xi}\\[10pt]
t = \xi \eta^2
\end{cases}
\end{align*}
and using the chain rule, we have
\begin{align}\label{eq:PDE-tPsi}
\eta^{-2}\tPsi-\xi = \xi^2(\xi-1)\frac{\partial^2}{\partial \xi^2}\tPsi + \xi^2\frac{\partial}{\partial \xi}\tPsi,
\end{align}
where
$$\tPsi=\tPsi(\xi,\eta):=\tG\big(1-\tfrac{1}{\xi},\xi \eta^2\big)=\tG(y,t).$$
Further, we observe that $\tPsi(\xi,\eta)$ is in $\mathbb{Z}[\xi][[\eta]]$, and therefore in $\mathbb{Z}[[\eta]][[\xi]]$.

Let us write
$$\tPsi(\xi,\eta):=\sum_{m\ge 0}\tpsi_m \xi^m,$$
where $\tpsi_m:=\tpsi_m(\eta)$. From \eqref{eq:PDE-tPsi}, we find that
\begin{gather*}
\tpsi_0=0,\\
\tpsi_1=\eta^2,
\end{gather*}
and for $m\ge 2$,
\begin{align*}
\tpsi_m=\frac{(m-1)^2\eta^2}{1+m(m-1)\eta^2}\tpsi_{m-1}.
\end{align*}
Therefore, for $m\ge 1$,
\begin{align*}
\tpsi_m=\frac{\big((m-1)!\big)^2\eta^{2m}}{\prod_{k=1}^m \big(1+k(k-1)\eta^2\big)}.
\end{align*}

We conclude that
\begin{align*}
\tPsi(\xi,\eta)=\sum_{m\ge 1}\frac{\big((m-1)!\big)^2\eta^{2m}}{\prod_{k=1}^m \big(1+k(k-1)\eta^2\big)}\xi^m,
\end{align*}
and therefore,
\begin{align}\label{eq:tG(y,t)}
\tG(y,t)=\sum_{m\ge 1}\frac{\big((m-1)!\big)^2 t^{m}}{\prod_{k=1}^m \big(1+k(k-1)(1-y)t\big)}.
\end{align}

\subsection{Proof of Theorem \ref{th:eo}}

We observe that
\begin{align*}
\sum_{n\ge 1}\sum_{[\pi]\in\fC_n^o}y^{\drop_{eo}([\pi])}t^n&=\sum_{n\ge 1}g_{2n}t^{2n}+\sum_{n\ge 1}g_{2n-1}t^{2n-1}\\
&=G(y,t^2)+t^{-1}\tG(y,t^2).
\end{align*}
Substituting \eqref{eq:G(y,t)} and \eqref{eq:tG(y,t)} into the above gives Theorem \ref{th:eo}.

\section{Two identities}

We close this paper with two interesting identities.

\begin{corollary}
	\begin{align}\label{eq:Iden-1}
	\sum_{m\ge 1}\frac{\big((m-1)!\big)^2t^{m}}{\prod_{k=1}^m (1+k^2 t)}=t
	\end{align}
	and
	\begin{align}\label{eq:Iden-2}
	\sum_{m\ge 1}\frac{m!(m-1)!t^{m}}{\prod_{k=1}^m \big(1+k(k+1)t\big)}=t.
	\end{align}
\end{corollary}

\begin{proof}
	For $n\ge 2$, we observe that each cycle in $\fC_{2n-1}^o$ has at least one odd-odd drop, in which the former entry is $2n-1$. Recall that
	$$\tF(x,t)=\sum_{n\ge 1}\sum_{[\pi]\in\fC_{2n-1}^o}x^{\drop_{oo}([\pi])}t^n.$$
	Therefore, $\tF(0,t)=t$. In light of \eqref{eq:tF(x,t)}, we arrive at \eqref{eq:Iden-1}.
	
	Also, for $n\ge 1$, each cycle in $\fC_{2n}^o$ has at least one even-odd drop, in which the former entry is $2n$. Since
	$$G(y,t)=\sum_{n\ge 1}\sum_{[\pi]\in\fC_{2n}^o}y^{\drop_{eo}([\pi])}t^n,$$
	we have $G(0,t)=0$. Recalling \eqref{eq:G(y,t)} yields \eqref{eq:Iden-2}.
\end{proof}

\subsection*{Acknowledgements}

The author was supported by a Killam Postdoctoral Fellowship from the Killam Trusts.

\bibliographystyle{amsplain}

\end{document}